\documentclass[12pt]{article}

\usepackage[a4paper, top=3cm, left=2.5cm, right=2.5cm, bottom=3cm]{geometry}
\usepackage{authblk}
\usepackage{amsmath, amsthm}
\usepackage{amssymb}
\usepackage{multicol,comment}
\usepackage{amsfonts}
\usepackage{graphicx, color, enumerate, fancyhdr, dsfont}

    \newcommand{\dom}{\mbox{\rm dom}}

    \newcommand{\Awf}{\mathcal{A}}

    \newcommand{\Hwf}{\mathcal{H}}

    \newcommand{\Mwf}{\mathcal{M}}
    \newcommand{\Nwf}{\mathcal{N}}
    
    \newcommand{\Pwf}{\mathcal{P}}

    \newcommand{\Swf}{\mathcal{S}}

    \newcommand{\bfrak}{\mathfrak{b}}
    \newcommand{\cfrak}{\mathfrak{c}}
    \newcommand{\dfrak}{\mathfrak{d}}

    \newcommand{\menos}{\smallsetminus}

    \newcommand{\add}{\mbox{\rm add}}
    \newcommand{\cov}{\mbox{\rm cov}}
    \newcommand{\non}{\mbox{\rm non}}
    \newcommand{\cof}{\mbox{\rm cof}}

    \newcommand{\Aor}{\mathds{A}}
    \newcommand{\Bor}{\mathds{B}}
    \newcommand{\Cor}{\mathds{C}}
    \newcommand{\Dor}{\mathds{D}}
    \newcommand{\Eor}{\mathds{E}}
    \newcommand{\Loc}{\mathds{LOC}}
    
    \newcommand{\Por}{\mathds{P}}
    \newcommand{\Qor}{\mathds{Q}}
    \newcommand{\Ror}{\mathds{R}}
    \newcommand{\Sor}{\mathds{S}}
    
    \newcommand{\Qnm}{\dot{\mathds{Q}}}

    \newcommand{\Abf}{\mathbf{A}}
    \newcommand{\Rbf}{\mathbf{R}}

    \newcommand{\Lc}{\mathbf{Lc}}

    \newcommand{\cf}{\mbox{\rm cf}}

    \newcommand{\leqT}{\preceq_{\mathrm{T}}}
    \newcommand{\eqT}{\cong_{\mathrm{T}}}
    \newcommand{\Dbf}{\mathbf{D}}

\title{A note on ``Another ordering of the ten cardinal characteristics in
Cicho\'n's Diagram'' and further remarks}
\author{Diego Alejandro Mej\'{\i}a}
\affil{Faculty of Science, Shizuoka University.}
\date{}

\newcommand{\Addresses}{{
  \bigskip
  \footnotesize

  Faculty of Science\par 
  Shizuoka University\par 
  836 Ohya, Suruga-ku, Shizuoka 422-8529\par
  JAPAN\par 
  E--mail address: \texttt{diego.mejia@shizuoka.ac.jp}
}}

\begin{document}

\makeatletter
\def\@roman#1{\romannumeral #1}
\makeatother

\newcounter{enuAlph}
\renewcommand{\theenuAlph}{\Alph{enuAlph}}

\theoremstyle{plain}
  \newtheorem{theorem}{Theorem}[section]
  \newtheorem{corollary}[theorem]{Corollary}
  \newtheorem{lemma}[theorem]{Lemma}
  \newtheorem{prop}[theorem]{Proposition}
  \newtheorem{claim}[theorem]{Claim}
  \newtheorem{exer}[theorem]{Exercise}
  \newtheorem{teorema}[enuAlph]{Theorem}
\theoremstyle{definition}
  \newtheorem{definition}[theorem]{Definition}
  \newtheorem{example}[theorem]{Example}
  \newtheorem{remark}[theorem]{Remark}
  \newtheorem{context}[theorem]{Context}
  \newtheorem{question}[theorem]{Question}
  \newtheorem{problem}[theorem]{Problem}
  \newtheorem{notation}[theorem]{Notation}

\maketitle

\newcommand{\la}{\langle}
\newcommand{\ra}{\rangle}
\newcommand{\id}{\mathrm{id}}
\newcommand{\sig}{\boldsymbol{\Sigma}}
\newcommand{\cosig}{\boldsymbol{\Pi}}

\newcommand{\leqdi}{\preceq_{\mathrm{di}}}
\newcommand{\eqdi}{\approx_{\mathrm{di}}}
\newcommand{\leqcdi}{\preceq_{\mathrm{cdi}}}
\newcommand{\eqcdi}{\approx_{\mathrm{cdi}}}
\newcommand{\eqPc}{\approx_{\mathrm{P}}}

\newcommand{\sbf}{\mathbf{s}}
\newcommand{\tbf}{\mathbf{t}}
\newcommand{\matit}{\mathbf{m}}
\newcommand{\Fr}{\mathrm{Fr}}

\newcommand{\Slm}{\mathbf{aLc}^*}
\newcommand{\Mg}{\mathbf{Mg}}

\newcommand{\ls}{\mathrm{ls}}
\newcommand{\trk}{\mathrm{trk}}
\newcommand{\Leb}{\mathrm{Leb}}
\newcommand{\Ct}{\mathbf{Ct}}
\newcommand{\Id}{\mathbf{Id}}

\begin{abstract}
   In this note, we relax the hypothesis of the main results in Kellner--Shelah--T\v{a}nasie's \emph{Another ordering of the ten cardinal characteristics in Cicho\'n's diagram}.
\end{abstract}

\section{Introduction}\label{SecIntro}

This work belongs to the framework of consistency results where several cardinal invariants are pairwise different, in particular those in Cicho\'n's diagram (Figure~\ref{cichon}). Though we assume that the reader is familiar with this diagram, a characterization of its cardinal characteristics is presented in Section~\ref{SecCardinv}. It is well-known that this diagram is \emph{complete} in the sense that no other inequality can be proved between two cardinal invariants there. See e.g.~\cite{BJ} for details and original references.

\begin{figure}
\begin{center}
\includegraphics{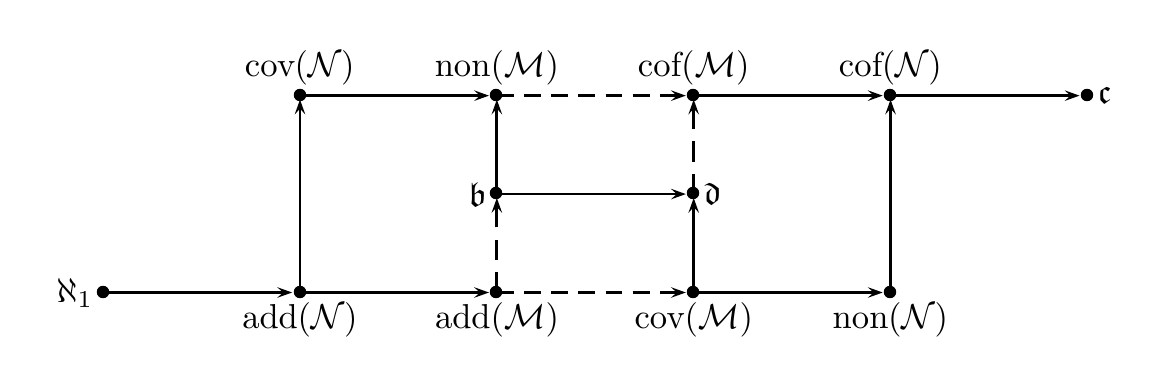}
\caption{Cicho\'n's diagram. The arrows represent $\leq$. The dashed arrows mean $\add(\Mwf)=\min\{\bfrak,\cov(\Mwf)\}$ and $\cof(\Mwf)=\max\{\dfrak,\cov(\Mwf)\}$.}
\label{cichon}
\end{center}
\end{figure}

Resent progress in this framework appears in~\cite{GMS,FFMM,mejiavert,KTT,GKS,KST,BCM,GKMS}. In particular, \cite{KTT} deals with the method of Boolean ultrapowers, used in~\cite{GKS} to separate all cardinals in Cicho\'n's diagram using four strongly compact cardinals. This result was improved in~\cite{BCM} where it is shown that three strongly compact cardinals suffice. In~\cite{GKMS} more classical cardinal invariants are separated simultaneously.

This survey focuses on Kellner--Shelah--T\v{a}nasie's work~\cite{KST}, where the consistency of an alternative order of the ten values of Cicho\'n's diagram is proved. Concretely, they show:
\begin{enumerate}[(A)]
  \item Assuming GCH, if $\lambda_1<\lambda_2<\lambda_3<\lambda_4<\lambda_5$ are regular $\aleph_1$-inaccessible cardinals,\footnote{A cardinal $\lambda$ is \emph{$\theta$-inaccessible} if $\mu^\nu<\lambda$ for any $\mu<\lambda$ and $\nu<\theta$.} then there is a ccc poset that forces $\add(\Nwf)=\lambda_1$, $\bfrak=\lambda_2$, $\cov(\Nwf)=\lambda_3$, $\non(\Mwf)=\lambda_4$, and $\cov(\Mwf)=\cfrak=\lambda_5$.
  \item Assuming GCH, if $\kappa_9<\lambda_1<\kappa_8<\lambda_2<\kappa_7<\lambda_3<\kappa_6<\lambda_4
      <\lambda_5<\lambda_6<\lambda_7<\lambda_8<\lambda_9$ are regular cardinals such that $\lambda_2$ is a successor cardinal, $\lambda_i$ is $\aleph_1$-inaccessible for $i=1,\ldots,5$ and $\kappa_j$ is strongly compact for $j=6,\ldots,9$, then there is a ccc poset that forces
      \begin{multline*}
    \add(\Nwf)=\lambda_1,\ \mathfrak{b}=\lambda_2,\ \cov(\Nwf)=\lambda_3,\ \non(\Mwf)=\lambda_4,\\
        \cov(\Mwf)=\lambda_5,\ \non(\Nwf)=\lambda_6,\ \mathfrak{d}=\lambda_7,\ \cof(\Nwf)=\lambda_8,\text{\ and }\mathfrak{c}=\lambda_9.
\end{multline*}
\end{enumerate}

In this work, we show how to relax many of the hypothesis on (A) and (B), for instance, replace GCH by very specific cardinal arithmetic conditions, show that the hypothesis ``$\aleph_1$-inaccessible'' is only required for $\lambda_3$, and that $\lambda_2$ could also be a limit (regular) cardinal in (B). To be precise, we prove:

\begin{teorema}\label{mainnoba}
   Let $\lambda_1\leq\lambda_2\leq\lambda_3\leq\lambda_4$ be regular uncountable cardinals, and let $\lambda_5=\lambda_5^{<\lambda_4}$ be a cardinal. Assume that either 
   \begin{enumerate}[(i)]
     \item $\lambda_2=\lambda_3$, or
     \item $\lambda_2=\lambda_2^{<\lambda_2}<\lambda_3$, $\lambda_3$ is $\aleph_1$-inaccessible, and either $\lambda_4=\lambda_5$ or $\lambda_4^{\aleph_0}<\lambda_5$.
   \end{enumerate} 
   Then there is a ccc poset forcing $\add(\Nwf)=\lambda_1$, $\bfrak=\lambda_2$, $\cov(\Nwf)=\lambda_3$, $\non(\Mwf)=\lambda_4$, and $\cov(\Mwf)=\cfrak=\lambda_5$.
\end{teorema}

\begin{teorema}\label{mainba}
  Assume that $\kappa_9<\lambda_1<\kappa_8<\lambda_2<\kappa_7<\lambda_3<\kappa_6<\lambda_4
      \leq\lambda_5<\lambda_6<\lambda_7<\lambda_8<\lambda_9$ such that
      \begin{enumerate}[(i)]
        \item $\lambda_1,\ldots,\lambda_5$ are as in the hypothesis of Theorem~\ref{mainnoba};
        \item for $j=6,\ldots,9$, $\kappa_j$ is strongly compact and $\lambda_j^{\kappa_j}=\lambda_j$; and
        \item for $j=6,7,8$, $\lambda_j$ is regular.
      \end{enumerate}
      then there is a ccc poset that forces
      \begin{multline*}
    \add(\Nwf)=\lambda_1,\ \mathfrak{b}=\lambda_2,\ \cov(\Nwf)=\lambda_3,\ \non(\Mwf)=\lambda_4,\\
        \cov(\Mwf)=\lambda_5,\ \non(\Nwf)=\lambda_6,\ \mathfrak{d}=\lambda_7,\ \cof(\Nwf)=\lambda_8,\text{\ and }\mathfrak{c}=\lambda_9.
\end{multline*}
\end{teorema}

These improvements does not demand much changes in the original proofs, actually, the only dramatic modification is the construction of small posets, this to avoid the hypothesis ``$\aleph_1$-inaccessible''. In relation to Theorem~\ref{mainba} (or the method of Boolean ultrapowers) we show how to weaken the notion $\mathrm{COB}(\theta,\lambda)$ used in \cite{KTT,GKS,KST} to the more natural ``$\theta$-dominating family of size $\lambda$'' (see Definition~\ref{Defrelsys}(7)), i.e., that Boolean ultrapowers preserves such dominating families (by just modifying its size in some cases). Details are presented in Lemmas~\ref{BUPpre}(g) and~\ref{BUPpres}(c).

\begin{remark}
  Any strict inequality after $\lambda_5$ can be replaced by $=$ arbitrarily, and in such case the corresponding strongly compact cardinal can be omitted. For example, if we let $\lambda_{j-1}=\lambda_{j}$ ($j=6,7,8,9$) then the strongly compact cardinal $\kappa_j$ can be omitted, while $\lambda_{9-j}\leq\lambda_{10-j}$ is allowed (here $\lambda_0=\aleph_1$).
  
  On the other hand, in~\cite{GKMS} it is proved that the large gaps between the cardinals on the left side of Cicho\'n's diagram can be reduced, while separating on the right side. However, the strongly compact cardinals are still required for this purpose.
\end{remark}

This paper is structured as follows. In Section~\ref{SecCardinv} we review general notions related to preservation of unbounded families. In Section~\ref{SecFAM} we summarize some technicalities used in~\cite{KST} without going too deep into details, e.g., we manage to avoid introducing FAM-limits, which is the core of the tools in that reference. In Section~\ref{SecMain} we prove Theorem~\ref{mainnoba} and in Section~\ref{SecBA} we present further remarks on the Boolean ultrapower method towards Theorem~\ref{mainba}.

\section{Relational systems and preservation}\label{SecCardinv}

In many cases, cardinal invariants of the continuum are defined through \emph{relational systems}.

\begin{definition}\label{Defrelsys}
  A \emph{relational system} is a triplet $\Abf=\la X,Y,\sqsubset\ra$ where $\sqsubset$ is a relation contained in $X\times Y$. For $x\in X$ and $y\in Y$, $x\sqsubset y$ is often read \emph{$y$ $\sqsubset$-dominates $x$}.

  Fix a cardinal $\theta$.
  \begin{enumerate}[(1)]
    \item A family $D\subseteq Y$ is \emph{$\Abf$-dominating} iff every member of $X$ is $\sqsubset$-dominated by some member of $D$.
    \item A family $F\subseteq X$ is \emph{$\Abf$-unbounded} if there is \underline{no} real in $Y$ that $\sqsubset$-dominates every member of $F$.
    \item The relational system $\Abf^\perp:=\la Y,X,\not\sqsupset\ra$ is referred to as the \emph{dual of $\Abf$}. Note that $(\Abf^\perp)^\perp=\Abf$ and that $F\subseteq X$ is $\Abf$-unbounded iff it is $\Abf^\perp$-dominating.
    \item For a set $M$ say that $y\in Y$ is \emph{$\Abf$-dominating over $M$} if $x\sqsubset y$ for all $x\in X\cap M$.
    \item Say that $x\in X$ is \emph{$\Abf$-unbounded over $M$} if it is $\Abf^\perp$-dominating over $M$, that is, $x\not\sqsubset y$ for all $y\in Y\cap M$.
    \item The cardinal $\bfrak(\Abf)$ denotes the least size of an $\Abf$-unbounded family and $\dfrak(\Abf)$ is the least size of an $\Abf$-dominating family. Note that $\bfrak(\Abf^\perp)=\dfrak(\Abf)$ and $\dfrak(\Abf^\perp)=\bfrak(\Abf)$.
    \item A family $D\subseteq Y$ is \emph{$\theta$-$\Abf$-dominating} if, for every $E\in[X]^{<\theta}$, there is some $\Abf$-dominating $y\in D$  over $E$.
    \item A family $F\subseteq X$ is \emph{$\theta$-$\Abf$-unbounded} if it is $\theta$-$\Abf^\perp$-dominating, that is, for any $H\in[Y]^{<\theta}$ there is some $\Abf$-unbounded $x\in F$ over $H$.
    \item A family $D\subseteq Y$ is \emph{strongly $\theta$-$\Abf$-dominating} if $|D|\geq\theta$ and, for every $x\in X$, $|\{y\in D:x\not\sqsubset y\}|<\theta$.
    \item A family $F\subseteq X$ is \emph{strongly $\theta$-$\Abf$-unbounded} if it is strongly-$\theta$-$\Abf^\perp$-dominating, that is, $|F|\geq\theta$ and, for every $y\in Y$, $|\{x\in F:x\sqsubset y\}|<\theta$.
  \end{enumerate}
\end{definition}

\begin{remark}\label{RemRelSys}
  Fix a relational system $\Abf=\la X,Y,\sqsubset\ra$.
  \begin{enumerate}[(1)]
    \item The cardinal invariants $\bfrak(\Abf)$ and $\dfrak(\Abf)$ may not always exist. Concretely, $\bfrak(\Abf)$ does not exist iff $\dfrak(\Abf)=1$. Dually, $\dfrak(\Abf)$ does not exists iff $\bfrak(\Abf)=1$.
    \item Any subset of $Y$ is $\Abf$-dominating iff it is $2$-$\Abf$-dominating. Likewise, $\Abf$-unbounded is equivalent to $2$-$\Abf$-unbounded.
    \item If $\theta\leq\theta'$ are cardinals, then any $\theta'$-$\Abf$-dominating family is $\theta$-$\Abf$-dominating. Likewise for unbounded families.
    \item If $\theta\geq 2$, then any $\theta$-$\Abf$-dominating family is $\Abf$-dominating.\footnote{Any subset of $Y$ is $0$-$\Abf$-dominating; and $D\subseteq Y$ is $1$-$\Abf$-dominating iff $D\neq\emptyset$.} Moreover, if $D\subseteq Y$ is $\theta$-$\Abf$-dominating family then $\dfrak(\Abf)\leq|D|$ and $\theta\leq\bfrak(\Abf)$. Similar statements hold for unbounded families, e.g., if $F\subseteq X$ is $\theta$-$\Abf$-unbounded then $\bfrak(\Abf)\leq|F|$ and $\theta\leq\dfrak(\Abf)$.
    \item Any strongly $\theta$-$\Abf$-dominating family is $\Abf$-dominating. Likewise for unbounded.
    \item If $\theta$ is regular and $D\subseteq Y$ is a strongly $\theta$-$\Abf$-dominating family, then $D$ is $|D|$-$\Abf$-dominating, so $\dfrak(\Abf)\leq|D|\leq\bfrak(\Abf)$. Similarly, if $F\subseteq X$ is strongly $\theta$-$\Abf$-unbounded then it is $|F|$-$\Abf$-unbounded, which implies $\bfrak(\Abf)\leq|F|\leq\dfrak(\Abf)$.
  \end{enumerate}
\end{remark}

Inequalities between cardinal invariants are often proved using the \emph{Tukey order} between relational systems. If $\Abf=\la X,Y,\sqsubset\ra$ and $\Abf'=\la X',Y',\sqsubset'\ra$ are relational systems, $\Abf\leqT\Abf'$ means that there are two maps $\varphi:X\to X'$ and $\psi:Y'\to Y$ such that, for any $x\in X$ and $y'\in Y'$, $\varphi(x)\sqsubset' y'$ implies $x\sqsubset\psi(y')$. In this case, the $\psi$-image of any $\Abf'$-dominating set is $\Abf$-dominating, and the $\varphi$-image of any $\Abf$-unbounded set is $\Abf'$-unbounded, thus $\bfrak(\Abf')\leq\bfrak(\Abf)$ and $\dfrak(\Abf)\leq\dfrak(\Abf')$. Say that $\Abf$ and $\Abf'$ are \emph{Tukey equivalent}, denoted by $\Abf\eqT\Abf'$, if $\Abf\leqT\Abf'$ and $\Abf'\leqT\Abf$.

Before we present examples, we review the preservation theory of unbounded families presented in \cite[Sect. 4]{CM}. This a generalization of Judah and Shelah's~\cite{JS} and Brendle's~\cite{Br} preservation theory.

\begin{definition}\label{DefgPrs}
Say that $\Rbf=\langle X,Y,\sqsubset\rangle$ is a \emph{generalized Polish relational system (gPrs)} if
\begin{enumerate}[(I)]
\item $X$ is a Perfect Polish space,
\item $Y=\bigcup_{e\in \Omega}Y_e$ where $\Omega$ is a non-empty set and, for some Polish space $Z$, $Y_e$ is non-empty and analytic in $Z$ for all $e\in \Omega$, and
\item $\sqsubset=\bigcup_{n<\omega}\sqsubset_{n}$ where $\langle\sqsubset_{n}: n<\omega\rangle$  is some increasing sequence of closed subsets of $X\times Z$ such that, for any $n<\omega$ and for any $y\in Y$,
$(\sqsubset_{n})^{y}=\{x\in X:x\sqsubset_{n}y \}$ is closed nowhere dense.
\end{enumerate}

If $|\Omega|=1$, we just say that $\Rbf$ is a \emph{Polish relational system (Prs)}.
\end{definition}

For the rest of this section, fix a gPrs $\Rbf=\langle X,Y,\sqsubset\rangle$.

\begin{remark}
  $\la X,\Mwf(X),\in\ra\leqT\Rbf$, so $\bfrak(\Rbf)\leq\non(\Mwf)$ and $\cov(\Mwf)\leq\dfrak(\Rbf)$.
\end{remark}

\begin{definition}
Let $\theta$ be a cardinal.
A poset $\Por$ is \textit{$\theta$-$\Rbf$-good} if, for any $\Por$-name $\dot{h}$ for a member of $Y$, there exists  a non-empty $H\subseteq Y$ (in the ground model) of size $<\theta$ such that, for any $x\in X$, if $x$ is $\Rbf$-unbounded over  $H$ then $\Vdash x\not\not\sqsubset \dot{h}$.

Say that $\Por$ is \textit{$\Rbf$-good} if it is $\aleph_1$-$\Rbf$-good.
\end{definition}

Note that $\theta<\theta'$ implies that any $\theta$-$\Rbf$-good poset is $\theta'$-$\Rbf$-good. Also, if $\Por\lessdot\Qor$ and $\Qor$ is $\theta$-$\Rbf$-good, then $\Por$ is $\theta$-$\Rbf$-good.

The notion of goodness is practical to preserve special types of unbounded families in generic extensions.

\begin{lemma}[{\cite[Lemma 4.7]{CM}}]\label{presunb}
Let $\theta$ be a regular cardinal, $\lambda\geq \theta$ a cardinal and let $\Por$ be a $\theta$-$\Rbf$-good poset.
\begin{enumerate}[(a)]
\item If $F\subseteq X$ is $\lambda$-$\Rbf$-unbounded, then $\Por$ forces that it is  $\dot{\lambda}'$-$\Rbf$-unbounded where, in the $\Por$-extension, $\dot{\lambda}'$ is the smallest cardinal $\geq\lambda$.
\item If $\cf(\lambda)\geq\theta$ and $F\subseteq X$ is strongly $\lambda$-$\Rbf$-unbounded then $\Vdash$``if $\lambda$ is a cardinal then $F$ is strongly $\lambda$-$\Rbf$-unbounded".
\item If $\dfrak(\Rbf)\geq\lambda$ then $\Por$ forces that $\dfrak(\Rbf)\geq\dot{\lambda}'$.
\end{enumerate}
\end{lemma}

As a first general example, every small poset is always good.

\begin{lemma}[{\cite[Lemma 2.7]{CM}}]\label{smallgood}
  If $\theta$ is a regular cardinal then any poset of size $<\theta$ is $\theta$-$\Rbf$-good. In particular, Cohen forcing $\Cor$ is $\Rbf$-good.
\end{lemma}

If $\theta$ is an uncountable regular cardinal then any FS support iteration of $\theta$-$\Rbf$-good $\theta$-cc posets is again $\theta$-$\Rbf$-good (and $\theta$-cc). Hence, according to the previous lemma, they preserve $\lambda$-$\Rbf$-unbounded families for any $\lambda\geq\theta$, and strongly $\lambda$-$\Rbf$-unbounded families for any $\lambda$ with $\cf(\lambda)\geq\theta$. Such unbounded families can be added using Cohen reals.

\begin{lemma}[{\cite[Lemma 4.15]{CM}}]\label{CohenUnb}
If $\nu$ is a cardinal with uncountable cofinality and $\mathbb{P}_{\nu}=\langle\mathbb{P}_{\alpha},\dot{\mathbb{Q}}_{\alpha}\rangle_{\alpha<\nu}$ is a FS iteration of non-trivial $\cf(\nu)$-cc posets, then $\Por_{\nu}$ adds a strongly $\nu$-$\Rbf$-unbounded family of size $\nu$.
\end{lemma}

\begin{theorem}[{\cite[Thm. 4.16]{CM}}]\label{ItPres}
Let $\theta$ be an uncountable regular cardinal, $\delta\geq\theta$ an ordinal, and let $\mathbb{P}_{\delta}=\langle\mathbb{P}_{\alpha},\dot{\mathbb{Q}}_{\alpha}\rangle_{\alpha<\delta}$ be a FS iteration such that, for each $\alpha<\delta$, $\dot{\Qor}_{\alpha}$ is a $\Por_{\alpha}$-name of a  non-trivial $\theta$-$\Rbf$-good $\theta$-cc poset. Then:
\begin{enumerate}[(a)]
    \item For any cardinal $\nu\in[\theta,\delta]$ with $\cf(\nu)\geq\theta$, $\Por_\nu$ adds a strongly $\nu$-$\Rbf$-unbounded family of size $\nu$ which is still strongly $\nu$-$\Rbf$-unbounded in the $\Por_\delta$-extension.
    \item For any cardinal $\lambda\in[\theta,\delta]$, $\Por_\lambda$ adds a $\lambda$-$\Rbf$-unbounded family of size $\lambda$ which is still $\lambda$-$\Rbf$-unbounded in the $\Por_\delta$-extension.
    \item $\mathbb{P}_{\delta}$ forces that $\bfrak(\Rbf)\leq\theta$ and  $|\delta|\leq\dfrak(\Rbf)$.
\end{enumerate}
\end{theorem}

In the practice, the converse inequalities of the last item are obtained by constructing a $\theta$-$\Rbf$-dominating family of size $|\delta|$.

\begin{example}\label{ExmInv}
We fix some notation. For any fuction $b:\omega\to V\menos\{\emptyset\}$ and $h\in\omega^\omega$, denote $\prod b:=\prod_{i<\omega}b(i)$ and $\Swf(b,h):=\prod_{i<\omega}[b(i)]^{\leq h(i)}$. For two functions $x,y$ with domain $\omega$, $x\in^* y$ denotes $\forall^\infty i<\omega(x(i)\in y(i))$. Denote by $\id_\omega$ the identity function on $\omega$.
  \begin{enumerate}[(1)]
      \item The relational system $\mathbf{D}=\la\omega^\omega,\omega^\omega,\leq^*\ra$ is a Prs, where $x\leq^* y$ means $\forall^\infty i<\omega(x(i)\leq y(i))$. Note that $\bfrak=\bfrak(\Dbf)$ and $\dfrak=\dfrak(\Dbf)$. Clearly, any $\omega^\omega$-bounding poset is $\mathbf{D}$-good. Any $\mu$-$\Fr$-linked poset (see Definition~\ref{Defuflinked}) is $\mu^+$-$\mathbf{D}$-good (\cite[Thm. 3.30]{mejiavert}).

      \item For $\mathcal{H}\subseteq\omega^\omega$ denote $\Lc(\omega,\mathcal{H}):=\langle\omega^\omega,\Swf(\omega,\mathcal{H}),\in^{*}\rangle$ where
          $\Swf(\omega,\mathcal{H}):=\bigcup_{h\in\mathcal{H}}\Swf(\omega,h)$ (here, $\omega$ denotes the constant function with value $\omega$). It is clear that $\Lc(\omega,\mathcal{H})$ is a gPrs. Any $\nu$-centered poset is $\nu^+$-$\Lc(\omega,\mathcal{H})$-good (\cite{JS}, see also \cite[Lemma 6]{Br} and \cite[Lemma 5.13]{BrM}).

          Whenever $\mathcal{H}$ is countable and non-empty, $\Lc(\omega,\mathcal{H})$ is a Prs because $\Swf(\omega,\mathcal{H})$ is $F_\sigma$ in $([\omega]^{<\omega})^\omega$. In addition, if $\mathcal{H}$ contains a function that goes to infinity then $\bfrak(\Lc(\omega,\mathcal{H}))=\add(\Nwf)$ and $\dfrak(\Lc(\omega,\mathcal{H}))=\cof(\Nwf)$ (see \cite[Thm. 2.3.9]{BJ}). Moreover, if all the members of $\Hwf$ go to infinity then any Boolean algebra with a strictly positive finitely additive measure is $\Lc(\omega,\Hwf)$-good (\cite{Ka}). In particular, any subalgebra of random forcing is $\mathbf{Lc}(\omega,\Hwf)$-good.

          For the rest of this paper, fix $\Hwf_*:=\{\id_\omega^{n+1}:n<\omega\}$.
      \item (Kamo and Osuga \cite{KO}) Fix a family $\mathcal{E}\subseteq\omega^\omega$ of size $\aleph_1$ of non-decreasing functions which satisfies
          \begin{enumerate}[(i)]
            \item $\forall e \in \mathcal{E}(e \leq \text{id}_\omega)$,
            \item $\forall e \in \mathcal{E}($ $\lim_{n\to+\infty}e(n)=+\infty$ and $\lim_{n\to+\infty}(n-e(n))=+\infty$),
            \item $\forall e \in \mathcal{E} \exists^{}e^{'} \in \mathcal{E}(e+1\leq^{*} e^{'})$ and
            \item $\forall \mathcal{E}^{'} \in [\mathcal{E}]^{\leq \aleph_0} \exists e\in \mathcal{E}\forall e^{'}\in\mathcal{E}^{'}(e^{'}\leq^{*}e)$.
          \end{enumerate}

          For $b,h \in \omega^\omega$ such that $b>0$ and $h\geq^*1$, we define $\hat{\Swf}(b,h)=\hat{\Swf}_{\mathcal{E}}(b,h)$ by
          \[\hat{\Swf}(b,h):=\bigcup_{e\in\mathcal{E}}\Swf(b,h^e)=\Big\{\varphi \in \prod_{n<\omega}\mathcal{P}(b(n)):\exists e\in \mathcal{E}\forall n<\omega (|\varphi(n)|\leq h(n)^{e(n)})\Big\}\]

          Let $n<\omega$. For $\psi, \varphi :\omega\to[\omega]^{<\omega}$,
          define the relation $\psi\blacktriangleright_{n}\varphi$ iff $\forall{k\geq n}(\psi(k)\nsupseteq\varphi(k))$, and define
          $\psi\blacktriangleright\varphi$ iff $\forall^{\infty}{k<\omega}(\psi(k)\nsupseteq\varphi(k))$, i.e., $\blacktriangleright=\bigcup_{n<\omega}\blacktriangleright_{n}$. Put $\Slm(b,h):=\langle \Swf(b,h^{\mathrm{id}_{\omega}}), \hat{\Swf}(b,h), \blacktriangleright\rangle$ which is a gPrs where $\Omega=\mathcal{E}$, $Z=\Swf(b,h^{\id_\omega})$ and $Y_e=\Swf(b,h^{e})$ for each $e\in\mathcal{E}$. Note that $Y_e$ is closed in $Z$.

          Any $\nu$-centered poset is $\nu^+$-$\Slm(b,h)$-good (\cite[Lemma 4.25]{CM}). On the other hand, if $b\nleq^*1$ then any $(h,b^{h^{\mathrm{id}_\omega}})$-linked poset (see \cite{KO}) is $2$-$\Slm(b,h)$-good (\cite[Lemma 10]{KO}).

          If $\sum_{i<\omega}\frac{h(i)^{i}}{b(i)}<\infty$ then $\Slm(b,h)\leqT\la\Nwf,2^\omega,\not\ni\ra$, so $\cov(\Nwf)\leq\bfrak(\Slm(b,h))$ and $\dfrak(\Slm(b,h))\leq\non(\Nwf)$ (see e.g. \cite[Lemma 2.3]{KM}).

      \item Denote $\Xi:=\{f:2^{<\omega}\to2^{<\omega} : \forall s\in 2^{<\omega}(s\subseteq f(s))\}$ and set $\Mg:=\la2^\omega,\Xi,\in^\bullet\ra$ where $x\in^\bullet f$ iff  $|\{s\in2^{<\omega}:x\supseteq f(s)\}|<\aleph_0$. This is a Prs and $\Mg\eqT\la2^\omega,\Mwf,\in\ra$, so $\bfrak(\Mg)=\non(\Mwf)$ and $\dfrak(\Mg)=\cov(\Mwf)$.

      \item Denote $\Id:=\la2^\omega,2^\omega,=\ra$, which is clearly a Prs. Note that $\bfrak(\Id)=2$ and $\dfrak(\Ct)=\cfrak:=2^{\aleph_0}$. It is easy to see that any $\theta$-cc poset is $\theta$-$\Ct$-good.
  \end{enumerate}
\end{example}

To finish this section, we review a general notion to preserve $\Dbf$-unbounded families.

\begin{definition}[{\cite{mejiavert}}]\label{Defuflinked}
Let $\Por$ be a poset and let $\mu$ be an infinite cardinal.
\begin{enumerate}[(1)]
    \item A set $Q\subseteq\Por$ is \emph{$\Fr$-linked} if, for any sequence $\bar{p}=\la p_n:n<\omega\ra$ in $Q$, there exists a $q\in\Por$ that forces $|\{n<\omega:p_n\in\dot{G}\}|=\aleph_0$.
    \item The poset $\Por$ is \emph{$\mu$-$\Fr$-linked} if $\Por=\bigcup_{\alpha<\mu}P_\alpha$ for some sequence $\la P_\alpha:\alpha<\mu\ra$ of $\Fr$-linked subsets of $\Por$. 
        When $\mu=\aleph_0$, we write \emph{$\sigma$-$\Fr$-linked}.
    \item The poset $\Por$ is \emph{$\mu$-$\Fr$-Knaster} if any subset of $\Por$ of size $\mu$ contains an $\Fr$-linked set of size $\mu$. 
\end{enumerate}
\end{definition}

\begin{theorem}[{\cite{BCM}}]\label{FrKnasterpresunb}
   If $\kappa$ is an uncountable regular cardinal then any $\kappa$-$\Fr$-Knaster poset preserves all the strongly $\kappa$-$\Dbf$-unbounded families from the ground model.
\end{theorem}

\section{Iteration candidates}\label{SecFAM}

The main issue in the construction of the FS iteration for Theorem~\ref{mainnoba} is to ensure that
strongly $\Dbf$-unbounded families are preserved. For this purpose, Kellner, T\v{a}nasie and Shelah~\cite{KST} use the method of FAM (finite additive measure) limits, originally introduced by Shelah~\cite{ShCov}. 

In the construction of the iteration, it is expected to use $\Eor$, the standard $\sigma$-centered poset to add an eventually different real, to control $\non(\Mwf)$, but it seems that they do not have FAM limits. For this reason, Kellner, T\v{a}nasie and Shelah used an alternative poset that behaves well with FAM limits.

\begin{definition}[{\cite[Def. 1.11 \& 1.13]{KST}}]
  Define the following functions by recursion on $n<\omega$.
  \begin{align*}
    \rho_*(n) & = \max\bigg\{\prod_{k<n}M_*(k),n+2\bigg\},\\
    \pi_*(n) & = \Big((n+1)^2\rho_*(n)^{n+1}\Big)^{\rho_*(n)^n},\\
    a_*(n) & = \pi_*(n)^{n+2},\\
    M_*(n) & = a_*(n)^2.
  \end{align*}
  For each $n<\omega$ define the norm $\mu^*_n:\Pwf(M_*(n))\to[0,\infty]$ by $\mu^*_n(x)=\log_{a_*(n)}\Big(\frac{M_*(n)}{M_*(n)-|x|}\Big)$ ($\mu^*_n(M_*(n))=\infty$).
  Set the tree $T^*:=\bigcup_{n<\omega}\prod_{k<n}M_*(k)$.

  Define $\tilde{\Eor}$, ordered by $\subseteq$, as the poset whose conditions are subtrees $p\subseteq T^*$ such that, for some $2\leq m<\omega$,
  \begin{enumerate}[(i)]
    \item $\trk(p)>3m$ where $\trk(p)$ is the stem of $p$,
    \item $\mu^*_t(p)\geq 1+\frac{1}{m}$ for any $t\in p$ above $\trk(p)$, where
      $\mu^*_t(p):=\mu_{|t|}\big(\{i\in M(|t|) : t^\frown i\in p\}\big)$.
  \end{enumerate}
  Denote $\ls(p):=\frac{1}{m}$ where $m$ is the maximal number satisfying (i) and (ii).
\end{definition}

\begin{lemma}[{\cite[Lemma 1.17]{KST}}]\label{Etildeprop}
   The forcing $\tilde{\Eor}$ satisfies the following properties.
   \begin{enumerate}[(a)]
     \item $(\rho_*,\pi_*)$-linked. In particular, it is $\sigma$-linked and $\Slm(b_*,\rho_*)$-good, where $b_*(n):=(n+1)\rho_*(n)^{n+1}$.
     \item It adds an eventually different real in $\omega^\omega$ over the ground model. In particular, it adds an $\Mg$-dominating real over the ground model.
     \item It is equivalent to some subalgebra of random forcing. In particular, it is $\Lc(\omega,\Hwf)$-good whenever $\Hwf$ is a countable infinite set of functions that diverge to infinity.
   \end{enumerate}
\end{lemma}

Now we present the basic structure of the FS iterations that are used to prove the main results.

\begin{definition}\label{DefItCand}
   Let $\theta$ be an uncountable regular cardinal.  A \emph{$\theta$-iteration candidate} $\tbf$ is composed of
   \begin{enumerate}[(I)]
     \item an ordinal $\delta_\tbf$ partitioned into two sets $S_\tbf,R_\tbf$ (some could be empty);
     \item a FS iteration $\la\Por_{\tbf,\alpha},\Qnm_{\tbf,\alpha}:\alpha<\delta_\tbf\ra$;
     \item whenever $\alpha\in S_\tbf$, $\Qnm_{\tbf,\alpha}$ is a $\Por_{\tbf,\alpha}$-name of a ccc poset of size $<\theta$;
     \item whenever $\alpha\in R_\tbf$, there is some $\Por'_{\tbf,\alpha}\lessdot\Por_{\tbf,\alpha}$ such that $\Qnm_{\tbf,\alpha}=\Sor^{V^{\Por'_{\tbf,\alpha}}}$ where $\Sor$ is either random forcing or $\tilde{\Eor}$.
   \end{enumerate}
   When the context is clear, the subindex $\tbf$ is omitted.
\end{definition}

Note that, whenever $\tbf$ is a $\theta$-iteration candidate and $\alpha\in S$, since $\Por_\alpha$ is ccc, we can find a cardinal $\mu_\alpha=\mu_{\tbf,\alpha}<\theta$ in the ground model such that $\Por_\alpha$ forces $|\Qnm_\alpha|\leq\mu_\alpha$, so we can put $\Qnm_\alpha=\{\dot{q}_{\alpha,\zeta}:\zeta<\mu_\alpha\}$.

For the rest of this section, fix an ordinal $\delta_*$, a partition $P:=\la S,R\ra$ of $\delta_*$, and a regular uncountable cardinal $\theta$. Denote $\Psi^\theta_{P}:=\prod_{\alpha<\delta_*}W^\theta_{P,\alpha}$ where
\[W^\theta_{P,\alpha}:=\left\{\begin{array}{ll}
    \theta & \text{if $\alpha\in S$},\\
    \omega^{<\omega}\times\{\frac{1}{n}:1\leq n<\omega\} & \text{if $\alpha\in R$}.
\end{array}\right.\]

The following result is essential to construct iteration candidates that serve our purposes.

\begin{theorem}[Engelking and Kar\l owicz {\cite{TopThm}}]\label{EngKarl}
   Assume that $\theta=\theta^{\aleph_0}$ and $\delta_*<(2^\theta)^+$. Then there is some $H^*\subseteq\Psi^\theta_{P}$ of size $\theta$ such that any countable partial function from $\Psi^\theta_{P}$ is extended by some member of $H^*$.
\end{theorem}

From now on, assume that $\theta=\theta^{\aleph_0}$, $\delta_*<(2^\theta)^+$ and that $H^*$ is as in the previous lemma.

Consider random forcing $\Bor$, ordered by $\subseteq$, whose conditions are trees $T\subseteq2^{<\omega}$ such that $\Leb([T]\cap[t])>0$ for any $t\in T$, where $\Leb$ denotes the Lebesgue measure on $2^\omega$. For $T\in\Bor$, denote by $\trk(T)$ the stem of $T$ and set $\ls(T)=\frac{1}{m}$ where $m$ is the maximal natural number such that $\Leb([T])>(1-\frac{1}{m})\Leb([\trk(T)])$.

The main point for the main results is to construct an iteration candidate that produces a $\theta$-$\Fr$-Knaster poset, this to ensure the preservation of strongly $\Dbf$-unbounded families. In the context of iterations, we need to look at quite uniform $\Delta$-systems of conditions to deal with such property. The corresponding technicalities are presented below.

\begin{definition}
Let $\tbf$ be a $\theta$-iteration candidate with $\delta_\tbf\leq\delta_*$, $S_\tbf\subseteq S$ and $R_\tbf\subseteq R$.
\begin{enumerate}[(1)]
   \item For $\alpha<\delta$, let $\Por^*_\alpha$ be the set of conditions $p\in\Por_\alpha$ such that, for any $\xi\in\dom p$, if $\xi\in S$ then $p(\xi)=\dot{q}_{\xi,\zeta}$ for some $\zeta<\mu_\xi$; if $\xi\in R$ then $p(\xi)$ is a $\Por'_\alpha$-name (not just a $\Por_\alpha$-name) and both $\trk(p(\xi))$ and $\ls(p(\xi))$ are already decided by $p{\upharpoonright}\xi$; and furthermore
       \[\sum_{\xi\in R\cap \dom p}\sqrt{\ls(p(\xi))}<\frac{1}{2}.\]
       It is shown in \cite[Lemma 2.34]{KST} that $\Por^*_\alpha$ is dense in $\Por_\alpha$.
   \item Say that $\bar p=\la p_i:i<\gamma\ra$ is a \emph{$\tbf$-uniform $\Delta$-system} if it satisfies:
     \begin{enumerate}[(i)]
       \item $p_i\in\Por^*_\delta$;
       \item there is some $m<\omega$ such that $\dom p_i=\{\alpha_{i,k}:k<m\}$ (increasing enumeration);
       \item there is some $v\subseteq m$ such that, for any $k\in v$, the sequence $\la\alpha_{i,k}:i<\gamma\ra$ is constant with value $\alpha_k$;
       \item if $k\in m\menos v$ then the sequence $\la\alpha_{i,k}:i<\gamma\ra$ is increasing (hence, $\la\dom p_i : i<\gamma\ra$ forms a $\Delta$-system with root $\{\alpha_k:k\in v\}$);
       \item there is a partition $m=s\cup r\cup e$ such that, for any $i<\gamma$ and $k<m$: $\alpha_{i,k}\in S$ iff $k\in s$; and $k\in r$ iff $\Qnm_{\alpha_{i,k}}=\Bor^{V^{\Por'_{\alpha_{i,k}}}}$;
       \item for any $k\in r\cup e$, the sequence $\la(\trk(p_i(\alpha_{i,k})),\ls(p_i(\alpha_{i,k}))):i<\gamma\ra$ has constant value $(t_k,\varepsilon_k)$;
       \item for any $k\in s\cap v$ there is some $\zeta_k$ such that $p_i(\alpha_{k})=\dot{q}_{\alpha_{k},\zeta_k}$ for any $i<\gamma$.
     \end{enumerate}
     It is not hard to see that a $\tbf$-uniform $\Delta$-system $\bar p$ determines a partial function $h_{\bar p}$ in $\Psi^\theta_P$ with domain $\bigcup_{i<\gamma}\dom p_i$ such that $h(\alpha_{i,k})=(t_k,\varepsilon_k)$ when $k\in r\cup e$, and $p_i(\alpha_{i,k})=\dot{q}_{\alpha_{i,k},\zeta_{h(\alpha_{i,k})}}$ when $k\in s$. Given $h\in H^*$, say that \emph{$\bar p$ follows $h$} if $h_{\bar p}\subseteq h$.
\end{enumerate}
\end{definition}

In the following result we consider the existence of a class $\Lambda^\theta_P$ of \emph{good} iteration candidates, whose elements are iterations that match our goals. This class is not formally defined in~\cite{KST} but instead it is clarified how such good iterations are constructed using FAM limits. In the result below, corresponding to the main technical lemma in~\cite{KST}, we try our best to summarize the important features of the good iterations without referring to the central notions of FAM limits at all.

\begin{theorem}[{\cite[Lemma 2.39]{KST}}]\label{mainlemma}
  There is a class $\Lambda^\theta_{P}$ of $\theta$-iteration candidates such that
  \begin{enumerate}[(a)]
     \item the trivial candidate $\tbf_0$ with $\delta_{\tbf_0}=0$ is in $\Lambda^\theta_P$;
     \item if $\tbf\in\Lambda^\theta_P$ then $\delta_\tbf\leq\delta_*$, $S_\tbf\subseteq S$ and $R_\tbf\subseteq R$;
     \item if $\tbf\in\Lambda^\theta_P$, $\delta=\delta_\tbf\in S$ and $\Qnm$ is a $\Por_{\delta}$-name of a ccc poset of size $<\theta$, then $\tbf^+\in\Lambda^\theta_P$ where $\tbf^+$ is the natural $\theta$-iteration candidate that extends $\tbf$ such that $\delta_{\tbf^+}=\delta+1$ and $\Qnm_{\tbf^+,\delta}=\Qnm$;
     \item if $\tbf\in\Lambda^\theta_P$, $\delta=\delta_\tbf\in R$, $A\subseteq\Por_{\delta}$ and $\Sor$ is either $\Bor$ or $\tilde{\Eor}$, then there is some $\tbf^+\in\Lambda^\theta_P$ extending $\tbf$ such that $A\subseteq\Por'_{\tbf^+,\delta}$, $|\Por'_{\tbf^+,\delta}|\leq\max\{|A|,\theta\}^{\aleph_0}$, and $\Qnm_{\tbf^+,\delta}=\Sor^{V^{\Por'_{\tbf^+,\delta}}}$;
     \item in addition to the above, if $\la\tbf^+_\eta: \eta<\omega_1\ra$ is a sequence form $\Lambda^\theta_P$ such that $\delta_{\tbf^+_\eta}=\delta+1$, $\tbf_\eta$ extends $\tbf$ (using the same $\Sor$ at $\delta$) and $\Por'_{\tbf^+_\xi,\delta}\subseteq\Por'_{\tbf^+_\eta,\delta}$ whenever $\xi<\eta<\omega_1$, then $\tbf^+\in\Lambda^\theta_P$ where $\tbf^+$ is the iteration candidate extending $\tbf$ such that $\delta_{\tbf^+}=\delta+1$, $\Por'_{\tbf^+,\delta}=\bigcup_{\eta<\omega_1}\Por'_{\tbf^+_\eta,\delta}$ and $\Qnm_{\tbf^+,\delta}$ is defined using the same $\Sor$.
     \item if $\delta\leq\delta_*$ is limit and $\la\tbf_\alpha:\alpha<\delta\ra$ is a sequence of iteration candidates in $\Lambda^\theta_P$ such that $\delta_{\tbf_\alpha}=\alpha$ and $\tbf_\beta$ extends $\tbf_\alpha$ whenever $\alpha<\beta<\delta$, then the natural direct limit of $\la\tbf_\alpha:\alpha<\delta\ra$ belongs to $\Lambda^\theta_P$;
     \item if $\tbf\in\Lambda^\theta_P$ and $\bar p=\la p_n:n<\omega\ra$ is a $\tbf$-uniform $\Delta$-system, then there is some $q\in\Por_{\delta_\tbf}$ that forces $|\{n<\omega:p_n\in\dot{G}\}|=\aleph_0$.
  \end{enumerate}
\end{theorem}

As a consequence, good iteration candidates are $\theta$-$\Fr$-Knaster, as desired.

\begin{corollary}\label{FrIt}
   If $\tbf\in\Lambda^\theta_P$ then $\Por_\delta$ is $\theta$-$\Fr$-Knaster.
\end{corollary}
\begin{proof}
   Assume that $\{p_i:i<\theta\}\subseteq\Por^*_\delta$. It is possible to find an increasing function $g:\theta\to\theta$ such that $\{p_{g(i)}:i<\theta\}$ forms a $\tbf$-uniform $\Delta$-system. Hence, by Theorem~\ref{mainlemma}(g), this family is $\Fr$-linked.
\end{proof}

\section{The main results}\label{SecMain}

Now we are ready to prove Theorem~\ref{mainnoba}. Like in~\cite{KST}, we require the additional hypothesis $\lambda_5\leq2^{\lambda_2}$ to make sense of Theorem~\ref{mainlemma} (because it relies on Theorem~\ref{EngKarl}), but afterwards it is shown how to get rid of this requirement.

The case $\lambda_2=\lambda_3$ of Theorem~\ref{mainnoba} can be solved by very standard methods from~\cite{Br} without good iteration candidates. The argument below can be imitated to prove this simpler case, just ignore (3) and (3') and, in (4) and (4'), use small models as in (1),(1') and (2),(2').

\begin{theorem}\label{main1}
   Assume that $\lambda_1\leq\lambda_2=\lambda_2^{\aleph_0}<\lambda_3\leq\lambda_4$ are uncountable regular cardinals, let $\lambda_5$ be a cardinal such that $\lambda_5=\lambda_5^{<\lambda_4}\leq2^{\lambda_2}$ and either $\lambda_4=\lambda_5$ or $\lambda_4^{\aleph_0}<\lambda_5$, and assume that $\lambda_3$ is $\aleph_1$-inaccessible. Then there is a FS iteration of ccc posets that adds
   \begin{enumerate}[(a)]
     \item a strongly $\theta$-$\Rbf_i$-unbounded family of size $\theta$ for any regular $\lambda_i\leq\theta\leq\lambda_5$ and $i<4$, where $\Rbf_0=\Id$, $\lambda_0=2$, $\Rbf_1=\Lc(\omega,\Hwf_*)$, $\Rbf_2=\Slm(b_*,\rho_*)$, and $\Rbf_3=\Dbf$;
     \item a strongly $\theta$-$\Rbf_i$-unbounded family of size $\theta$ for any regular $\theta\in\{\lambda_4\}\cup(\lambda_4^{\aleph_0},\lambda_5]$, where $\Rbf_4=\Mg$;
     \item a $\lambda_i$-$\Rbf_i$-dominating family of size $\lambda_5$ for any $i<5$.
   \end{enumerate}
   In particular, this poset forces $\add(\Nwf)=\lambda_1$, $\bfrak=\lambda_2$, $\cov(\Nwf)=\lambda_3$, $\non(\Mwf)=\lambda_4$, and $\cov(\Mwf)=\cfrak=\lambda_5$.
\end{theorem}
\begin{proof}
    Let $W$ be the model obtained after adding $\lambda_5$-many Cohen reals. Fix $\delta_*=\lambda_5\lambda_4$ (ordinal product) and a partition $\la S_k : 1\leq k\leq 4\ra$ of $\lambda_5$. Put $S:=\{\lambda\beta+\rho: \beta<\lambda_4,\ \rho\in S_1\cup S_2\}$, $R:=\{\lambda\beta+\rho:\beta<\lambda_4,\ \rho\in S_3\cup S_4\}$ and $P:=\la S,R\ra$. Using Theorem~\ref{mainlemma}, we construct an iteration candidate $\tbf\in\Lambda^{\lambda_2}_P$ of length $\delta_*$ such that, for $\xi=\lambda_5\beta+\rho$ with $\beta<\lambda_4$ and $\rho<\lambda_5$:
   \begin{enumerate}[(1)]
     \item if $\rho\in S_1$ then $\Qnm_\xi=\Loc^{N_\xi}$ (localization forcing) where $N_\xi$, in $W_{\lambda_5\beta}=W^{\Por_{\lambda_5\beta}}$, is a transitive model of (a large fragment of) ZFC of size ${<}\lambda_1$;
     \item if $\rho\in S_2$ then $\Qnm_\xi=\Dor^{N_\xi}$ (Hechler forcing) where $N_\xi$, in $W_{\lambda_5\beta}$, is a transitive model of ZFC of size ${<}\lambda_2$;
     \item if $\rho\in S_3$ then $\Por'_\xi\lessdot\Por_\xi$ has size ${<}\lambda_3$ and $\Qnm_\xi=\Bor^{W^{\Por'_\xi}}$; and
     \item if $\rho\in S_4$ then $\Por'_\xi\lessdot\Por_\xi$ has size ${\leq}\lambda_4^{\aleph_0}$ and $\Qnm_\xi=\tilde{\Eor}^{W^{\Por'_\xi}}$.
   \end{enumerate}
   Furthermore, the models $N_\xi$ and the subposets $\Por'_\xi$ are constructed such that, for any $\beta<\lambda_4$:
   \begin{enumerate}[(1')]
     \item if $F\in W_{\lambda_5\beta}$ is a subset of $\omega^\omega$ of size ${<}\lambda_1$, then there is some $N_\xi$ with $\xi=\lambda_5\beta+\rho$ (as above) such that $\rho\in S_1$ and $F\subseteq N_\xi$;
     \item if $F\in W_{\lambda_5\beta}$ is a subset of $\omega^\omega$ of size ${<}\lambda_2$, then there is some $N_\xi$ with $\xi=\lambda_5\beta+\rho$ such that $\rho\in S_2$ and $F\subseteq N_\xi$;
     \item if $\Awf\in W_{\lambda_5\beta}$ is a family of Borel null sets of size ${<}\lambda_3$, then there is some $\xi=\lambda_5\beta+\rho$ such that $\rho\in S_3$ and all the members of $\Awf$ are coded in $W^{\Por'_\xi}$;
     \item if $F\in W_{\lambda_5\beta}$ is a subset of $\omega^\omega$ of size ${<}\lambda_4$, then there is some $\xi=\lambda_5\beta+\rho$ such that $\rho\in S_4$ and $F\subseteq W^{\Por'_\xi}$.
   \end{enumerate}
   Using standard counting arguments and Theorem~\ref{mainlemma}, we can construct such a $\tbf$.
   We prove that $\Por:=\Cor_{\lambda_5}\ast\Por_{\tbf,\delta_*}$ is as required. For (a), we show that the first $\theta$-many Cohen reals form a strongly $\theta$-$\Rbf_i$-unbounded family, in the $\Por$-extension, for any regular $\lambda_i\leq\theta\leq\lambda_5$. The case $i=0$ is quite obvious; when $i=1,3$, $\Por$ is $\lambda_i$-$\Rbf_i$-good (see Lemmas~\ref{smallgood} and~\ref{Etildeprop}, and Example~\ref{ExmInv}), so the desired family is strongly unbounded by Theorem~\ref{ItPres}; when $i=2$, see Theorems~\ref{mainlemma} and~\ref{FrKnasterpresunb}, and Corollary~\ref{FrIt}.

   For (b), since the iteration has cofinality $\lambda_4$, it adds a cofinal sequence of Cohen reals of length $\lambda_4$, which is a strongly $\lambda_4$-$\Rbf_4$-unbounded family. Now, if $\lambda_4^{\aleph_0}<\lambda_5$, then all the iterands are posets of size $\leq\lambda_4^{\aleph_0}$, hence $\Por$ is $\big(\lambda_4^{\aleph_0}\big)^{+}$-$\Rbf_4$-good (see Lemma~\ref{smallgood}). Thus, for any regular $\lambda_4^{\aleph_0}<\theta\leq\lambda_5$, the first $\theta$-many Cohen reals form a strongly $\theta$-$\Rbf_4$-unbounded family by Theorem~\ref{ItPres}.

   We finish with (c). The case $i=0$ is again trivial (witnessed by the family of all the reals in the final extension). The other cases are similar, we show the case $i=2$ as an example. By (2) and (2'), it is easy to see that $\{\dot{d}_{\beta,\rho}:\beta<\lambda_4,\ \rho\in S_2\}$ forms a $\lambda_2$-$\Rbf_2$-dominating family where each $\dot{d}_{\beta,\rho}$ is the $\Dbf$-dominating real over $N_\xi$ (with $\xi=\lambda_5\beta+\rho$) added by $\Qnm_\xi$.
\end{proof}

\begin{theorem}\label{main2}
    If $\lambda_2^{<\lambda_2}=\lambda_2$ then the assumption $\lambda_5\leq 2^{\lambda_2}$ in the previous theorem can be omitted.
\end{theorem}
\begin{proof}
   Assume $\lambda_2^{<\lambda_2}=\lambda_2$ and $2^{\lambda_2}<\lambda_5$.
   Let $\Ror:=\mathrm{Fn}_{<\lambda_2}(\lambda_5\times\lambda_2,2)$. Note that $\Ror$ is ${<}\lambda_2$-closed and has $\lambda_2$-cc, so it preserves cofinalities and forces that the hypothesis of the previous theorem holds.

   Denote the ground model by $V$ and $V':=V^\Ror$. It is clear that any ccc poset in $V$ is also ccc in $V'$. On the other hand, by Easton's Lemma, if $\Qor$ is a ccc poset in $V$ then $\Ror$ is ${<}\lambda_2$-distributive in $V^{\Qor}$, hence $\omega^\omega\cap V^{\Ror\times\Qor}=\omega^\omega\cap V^\Qor$.

   In $V$, define $\delta_*$, $\la S_k:1\leq k \leq 4\ra$ and $P=\la S,R\ra$ as in the previous proof.
   In $W:=V^{\Cor_{\lambda_5}}$, by recursion, we define an iteration candidate $\tbf$ of length $\delta_*$ such that (1)--(4) and (1')--(4') of the previous proof hold, and such that $\Vdash_{\Ror^V}^W\tbf\in\Lambda^{\lambda_2}_{P}$. Though we construct $\tbf$ in $W$, Theorem~\ref{mainlemma} is used in $W':=V^{\Ror\times\Cor_{\lambda_5}}$ to ensure that $\tbf\in\Lambda^{\lambda_2}_{P}$. Let $\xi=\lambda_5\beta+\rho$ where $\beta<\lambda_4$ and $\rho<\lambda_5$, and assume that $\tbf$ has been constructed up to $\xi$. Limit steps are easy since it is enough to take the direct limit; when $\rho\in S$, we can define $\Qnm_\xi$ arbitrarily as in (1) and (2) of the previous proof; the interesting case is when $\rho\in R$.

   We show the case $\rho\in S_3$ (the case $\rho\in S_4$ is similar). Start with $\Por^0_\xi\lessdot\Por_\xi$ of size ${<}\lambda_3$ containing a small family of codes as in (3') of the previous proof. By applying Theorem~\ref{mainlemma}(d) in $W'$, there is some $\Por^1_\xi\lessdot\Por_\xi$ of size ${<}\lambda_3$ such that $\Por^0_\xi\subseteq\Por^1_\xi$ and that further iterating with $\Bor^{{W'}^{\Por^1_\xi}}$ yields an iteration candidate in $\Lambda^{\lambda_2}_P$. Recall that the reals in $W'_\xi:={W'}^{\Por_\xi}$ are the same as those in $W_\xi:=W^{\Por_\xi}$. Therefore, in $W$, since $\Ror^V$ has $\lambda_3$-cc, we can find $\Por^2_\xi\lessdot\Por_\xi$ of size ${<}\lambda_3$ such that, in $W'$, $\Por^1_\xi\subseteq\Por^2_\xi$. Repeating this argument (and taking unions at limit stages), construct a $\subseteq$-increasing sequence $\la\Por^\eta_\xi:\eta<\omega_1\ra$ of complete subposets of $\Por_\xi$ of size ${<}\lambda_3$ such that $\la\Por^{2\eta}_\xi:\eta<\omega_1\ra\in W$ and, in $W'$, when the iteration so far is extended one step by $\Bor^{{W'}^{\Por^{2\eta+1}_\xi}}$, it is in $\Lambda^{\lambda_2}_P$. Put $\Por'_\xi:=\bigcup_{\eta<\omega_1}\Por^{2\eta}_\xi$. Thus, $\Por'_\xi\lessdot\Por_\xi$ belongs to $W$, it has size ${<}\lambda_3$, and by Theorem~\ref{mainlemma}(e) the iteration extended by $\Bor^{{W'}^{\Por'_\xi}}=\Bor^{W^{\Por'_\xi}}$ (though in $W$) belongs to $\Lambda^{\lambda_2}_P$ in the model $W'$. This finishes the construction.

   The same argument to prove (a), (b) and (c) in the previous theorem works, with the exception of the strongly $\theta$-$\Dbf$-unbounded families for any regular $\theta\in[\lambda_2,\lambda_5]$. In $W'$, like in the previous proof, we have that the first $\theta$-many Cohen reals $\{ c_\alpha:\alpha<\theta\}$ (which are in $W$) form a strongly $\theta$-$\Dbf$-unbounded family in $W'_{\delta_*}$. It remains to show that the same holds in $W_{\delta_*}$. Let $y\in\omega^\omega\cap W_{\delta_*}$. In $W'_{\delta_*}$, we have that $|\{\alpha<\theta:c_\alpha\leq^* y\}|<\theta$, and since this set is in $W_{\delta_*}$, the same holds in this model.
\end{proof}

\section{Further remarks on Boolean ultrapowers}\label{SecBA}

The following result summarizes the properties of Boolean ultrapowers used in~\cite{KTT,GKS,KST} to separate cardinal characteristics, with the exception of property (g). The latter is an observation of the author.

\begin{lemma}\label{BUPpre}
   Assume that $\kappa$ is a strongly compact cardinal and that $\theta$ is a cardinal such that
   $\theta^\kappa=\theta$. Then there is a complete embedding $j:V\to M$ such that:
   \begin{enumerate}[(a)]
     \item $M$ is transitive and $M^{{<}\kappa}\subseteq M$.
     \item $\kappa=\min\{\alpha\in\mathbf{On}: j(\alpha)>\alpha\}$ (the critical point of $j$).
     \item For any cardinal $\lambda\geq\kappa$ such that either $\lambda\leq\theta$ or $\lambda^\kappa=\lambda$, we have $\max\{\lambda,\theta\}\leq j(\lambda)<\max\{\lambda,\theta\}^+$.
     \item If $a$ is a set of size ${<\kappa}$ then $j(a)=j[a]$.
     \item If $\lambda>\kappa$ and $I$ is a ${<}\lambda$-directed partial order, then $j[I]$ is cofinal in $j(I)$.
     \item If $\cf(\alpha)\neq\kappa$ then $\cf(j(\alpha))=\cf(\alpha)$.
     \item If $\Abf=\la X,Y,\subseteq\ra$ is a relational system and $D\subseteq Y$ is $\lambda$-$\Abf$-dominating, then
         \begin{enumerate}[(i)]
           \item if $\lambda<\kappa$ then $j(D)$ is $\lambda$-$j(\Abf)$-dominating;
           \item if $\kappa<\lambda$ then $j[D]$ is $\lambda$-$j(\Abf)$-dominating.
         \end{enumerate}
   \end{enumerate}
\end{lemma}
\begin{proof}
   Let $\Aor$ be the Boolean completion of the poset $\mathrm{Fn}_{<\kappa}(\theta,\kappa)$. As in \cite{KTT}, there is an ultrafilter on $\Aor$ such that its Boolean ultrapower yields a complete embedding $j:V\to M$ satisfying (a)--(f). We show (g) (though technicalities are referred to \cite{KTT}). Let $E\subseteq j(X)$ be a set of size ${<}\lambda$. If $\lambda<\kappa$ then $E\in M$ by (a), and since $M\models$``$j(D)$ is $\lambda$-$j(\Abf)$-dominating'', we can find a $z\in j(Y)$ such that $w\, j(\sqsubset)\, z$ for any $w\in E$. This shows that $j(D)$ is $\lambda$-$j(\Abf)$-dominating.

   Now assume $\kappa<\lambda$. Any $w\in E$ is a mixture of $\kappa$-many possibilities $\la x^w_\alpha:\alpha<\kappa\ra\in X^\kappa$. Now, since $\{x^w_\alpha: w\in E,\ \alpha<\kappa\}$ has size $<\lambda$, there is some $y\in D$ $\Abf$-dominating over this set. This implies that $j(y)$ is $j(\Abf)$-dominating over $E$, which shows that $j[D]$ is $\lambda$-$j(\Abf)$-dominating.
\end{proof}

This method of Boolean ultrapowers works to preserve and modify strong unbounded families and also certain type of dominating families. Though in the original framework from~\cite{KTT,GKS,KST} they preserve a very strong type of dominating family (the property $\mathrm{COB}$), we show in item (c) that a weaker type of dominating families is preserved.

\begin{lemma}\label{BUPpres}
   Assume that $j:V\to M$ is a complete embedding satisfying (a), (b), (d)--(g) of the previous lemma. Let $\Por$ be a $\nu$-cc poset with $\nu<\kappa$ uncountable regular, and let $\Abf=\la X,Y,\sqsubset\ra$ be a relational system where $X,Y$ are analytic (or co-analytic) subsets of some Polish space, and $\sqsubset\, \subseteq X\times Y$ is also analytic (or co-analytic). Then
   \begin{enumerate}[(a)]
     \item \emph{\cite{KTT,GKS}} $j(\Por)$ is $\nu$-cc.
     \item \emph{\cite{KTT,GKS}} If $\mu$ is regular and $\Por$ adds a strongly $\mu$-$\Abf$-unbounded family of size $\mu$, then so does $j(\Por)$.
     \item If $\Por$ adds a $\lambda$-$\Abf$-dominating family of size ${\leq}\chi$ then
       \begin{enumerate}
         \item if $\lambda<\kappa$ then $j(\Por)$ adds a $\lambda$-$\Abf$-dominating family of size ${\leq}|j(\chi)|$;
         \item if $\kappa<\lambda$ then $j(\Por)$ adds a $\lambda$-$\Abf$-dominating family of size ${\leq}|\chi|$.
       \end{enumerate}
   \end{enumerate}
\end{lemma}
\begin{proof}
   Property (a) is immediate from Lemma~\ref{BUPpre}(a). For (b), see e.g. \cite{GKMS}. We show (c). Let $X^*$ be the set of nice $\Por$-names of members of $X$, and set $Y^*$ likewise. Note that $j(X^*)$ coincides with the set of nice $j(\Por)$-names of members of $X$, likewise for $j(Y^*)$ (by Lemma~\ref{BUPpre}(a)). On the other hand, $j[X^*]$ is equal to the set of nice $j(\Por)$-names of members of $j[X]$, likewise for $j[Y^*]$. For $\dot{x}\in X^*$ and $\dot{y}\in Y^*$, define $\dot x\sqsubset^* \dot y$ iff $\Vdash_\Por \dot x\sqsubset \dot y$, and set $\Abf^*:=\la X^*,Y^*,\sqsubset\ra$. Note that, for $\dot u\in j(X^*)$ and $\dot v\in j(Y^*)$, $\dot u\, j(\sqsubset^*)\, \dot v$ iff $\Vdash_{j(\Por)}\dot u \sqsubset\dot v$.

   Assume that $\dot{D}=\{\dot{y}_\eta : \eta<\chi\}$ is forced by $\Por$ to be a $\lambda$-$\Abf$-dominating family. It is easy to see that $\dot{D}$ is $\lambda$-$\Abf^*$-dominating. Hence, Lemma~\ref{BUPpre}(g) applies: if $\lambda<\kappa$ then $j(\dot{D})$ is $\lambda$-$j(\Abf^*)$-dominating, which means that $j(\Por)$ forces that $j(\dot{D})$ is $\lambda$-$\Abf$-dominating (of size ${\leq}|j(\chi)|$); if $\kappa<\lambda$ then $\{j(\dot{y}_\eta):\eta<\chi\}$ is $\lambda$-$j(\Abf^*)$-dominating, which means that $j(\Por)$ forces that $\{j(\dot{y}_\eta):\eta<\chi\}$ is $\lambda$-$\Abf$-dominating (of size ${\leq}|\chi|$).
\end{proof}

\begin{theorem}
  If $\aleph_1<\kappa_9<\lambda_1<\kappa_8<\lambda_2<\kappa_7<\lambda_3<\kappa_6<\lambda_4
  \leq\lambda_5<\lambda_6<\lambda_7<\lambda_8<\lambda_9$ such that
    \begin{enumerate}[(i)]
        \item for $j=6,7,8,9$, $\kappa_j$ is strongly compact and $\lambda_j^{\kappa_j}=\lambda_j$,
        \item $\lambda_i$ is regular for $i\neq 5,9$,
        \item $\lambda_2^{<\lambda_2}=\lambda_2$, $\lambda_4^{\aleph_0}=\lambda_4$, $\lambda_5^{<\lambda_4}=\lambda_5$, and
        \item $\lambda_3$ is $\aleph_1$-inaccessible,
    \end{enumerate}
    then there is a ccc poset that forces
\begin{multline*}
    \add(\Nwf)=\lambda_1,\ \mathfrak{b}=\lambda_2,\ \cov(\Nwf)=\lambda_3,\ \non(\Mwf)=\lambda_4,\\
        \cov(\Mwf)=\lambda_5,\ \non(\Nwf)=\lambda_6,\ \mathfrak{d}=\lambda_7,\ \cof(\Nwf)=\lambda_8,\text{\ and }\mathfrak{c}=\lambda_9.
\end{multline*}
\end{theorem}
\begin{proof}
   Use the ccc poset constructed in Theorem~\ref{main2} and apply Boolean ultrapowers accordingly as in \cite{KST}. We omit all the details, but we remark that, instead of property $\mathrm{COB}$, we can just preserve or modify dominating families thanks to Lemma~\ref{BUPpres}(c).
\end{proof}

\begin{remark}
  In~\cite{KST}, also for the corresponding results in~\cite{KTT,GKS,BCM}, the authors put quite an effort to construct (strong) dominating families satisfying $\mathrm{COB}$ through the iteration. But thanks to Lemma~\ref{BUPpres}(c), we can work with a weaker and more natural type of dominating family, so the construction becomes as simple as in Theorem~\ref{main1}.
\end{remark}

\subsection*{Acknowledgements}
This work was supported by the Grant-in-Aid for Early Career Scientists 18K13448, Japan Society for the Promotion of Science.

This paper was developed for the conference proceedings corresponding to the Set Theory Workshop that Professor Hiroaki Minami organized in November 2018. The author is very thankful to Professor Minami for letting him participate in such wonderful workshop.

The author also thanks Martin Goldstern and Jakob Kellner for valuable discussions.

{\footnotesize
\bibliography{bibli}
\bibliographystyle{alpha}
}

\Addresses


\end{document}